\documentclass[10pt,leqno]{amsart}
\usepackage{amscd}
\usepackage{color}
\usepackage{amssymb}
\usepackage{amsfonts}
\usepackage{latexsym}
\usepackage{verbatim}
\usepackage{dsfont}

\theoremstyle{plain}
\newtheorem{theorem}{Theorem}[section]
\newtheorem{definition}[theorem]{Definition}
\newtheorem{lemma}[theorem]{Lemma}
\newtheorem{prop}[theorem]{Proposition}

\newtheorem{rem}[theorem]{Remark}

\renewcommand{\b}{\begin{equation}}
\newcommand{\e}{\end{equation}}

\newcommand\C{{\mathds{C}}}
\newcommand\R{{\mathds{R}}}

\title{On the J-flow in Sasakian manifolds
 }
\author[L. Vezzoni, M. Zedda]{Luigi Vezzoni, Michela Zedda}

\date{\today}
\address{Dipartimento di Matematica \lq\lq Giuseppe Peano\rq\rq \\ Universit\`a di Torino\\
Via Carlo Alberto 10\\
10123 Torino\\ Italy} 
 \email{luigi.vezzoni@unito.it; michela.zedda@gmail.com}

\thanks{This work was supported by the project FIRB ``Geometria differenziale e teoria geometrica delle funzioni'',
 the project PRIN
\lq\lq  Variet\`a reali e complesse: geometria, topologia e analisi armonica" and by G.N.S.A.G.A. of I.N.d.A.M} 

\begin{document}

\maketitle
\begin{abstract}
We study the space of Sasaki metrics on a compact manifold $M$ by introducing an odd-dimensional analogue of the $J$-flow. 
That leads to the notion of  {\em critical metric} in the Sasakian context. In analogy to the K\"ahler case, on a polarised Sasakian manifold there exists at most one {\em normalised} critical metric. The flow is a tool for texting the existence of such a metric.      
We show that some results proved by Chen in \cite{chen} can be generalised to the Sasakian case. In particular, the {\em Sasaki $J$-flow} is a gradient flow which has always a long-time solution minimising the distance on the space of Sasakian potentials of a polarized Sasakian manifold. The flow minimises an energy  functional whose definition depends on the choice of a background transverse K\"ahler form $\chi$. 
When $\chi$ has nonnegative transverse holomorphic bisectional curvature, the flow converges to a critical Sasakian structure.  
\end{abstract}

\section{Introduction}
Sasakian manifolds are the odd-dimensional counterpart of K\"ahler manifolds and are defined as odd-dimensional Riemannian manifolds $(M,g)$ whose Riemannian cone $(M\times\R^+,t^2g+dt^2)$ admits a K\"ahler structure. These manifolds are important for both geometric and physical reasons.  
In geometry they can be used  to produce new examples of complete K\"ahler manifolds, manifolds with special holonomy and Einstein metrics. Moreover, Sasakian manifolds play a role in the study of orbifolds since many K\"ahler orbifolds can be desingolarised by using Sasakian spaces. In theoretical physics these manifolds play a central role in the AdS/CFT correspondence (see e.g. \cite{CLPP,GM,GM1,MS,superstring1,MSY,superstring2}). We refer to \cite{bgbook,sparks} for general theory and recent advanced in the study of these manifolds.
%

Given  a Sasakian manifold, the choice of a K\"ahler structure on the Riemannian cone determines 
a unitary Killing vector field $\xi$ of the metric $g$ and an endomorphism  $\Phi$ of the tangent bundle to  $M$ such that    
$$
\Phi^2=-{\rm Id}+\eta\otimes \xi\,,\quad g(\Phi\cdot,\Phi\cdot)=g(\cdot,\cdot)-\eta\otimes \eta\,, \quad g=\frac12 d\eta\circ  ({\rm Id}\otimes \Phi)+\eta\otimes \eta\,,
$$
$\eta$ being the $1$-form dual to $\xi$ via $g$. It turns out that $\eta$ is a contact form and that $\Phi$ induces a CR-structure $(\mathcal D,J)$ on $M$. Moreover, $\Phi(X)={\rm D}_X\xi$ for every vector field $X$ on $M$, where ${\rm D}$ is the Levi-Civita connection of $g$. 
The quadruple $(\xi,\Phi,\eta,g)$ is usually called a {\em Sasakian structure} and the pair $(\xi,J)$ can be seen as a polarization of $M$.

The research of this paper is mainly motivated by \cite{bgs,guanzhang,guanzhangA,he} where it is approached the study of Riemannian and symplectic aspects of the space of Sasakian potentials 
$\mathcal{H}$ on a polarised Sasakian manifold. Our approach consists in using  an analogue of the $J$-flow in the context of Sasakian Geometry obtaining some results similar to the ones proved in  the K\"ahler case by Chen in \cite{chen}. The $J$-flow is a gradient geometric flow of K\"ahler structures introduced and firstly studied  by Donaldson in \cite{donaldsonNC} from the point of view of moment maps and by Chen in \cite{chen} in relation to the Mabuchi energy. It is defined as the gradient flow of a functional $J_{\chi}$ defined on the space of {\em normalized} K\"ahler potentials whose definition depends on a fixed background K\"ahler structure $\chi$.  
Chen proved in \cite{chen} that the flow has always a unique long time solution which,  in the special case when $\chi$ has nonnegative biholomorphic curvature, converges to a critical K\"alher metric. Further results about the flow are obtained in \cite{LSS,songweinkoveJ,wein1,wein2}. 

As far as we know, the interest for geometric flows in foliated manifolds comes from \cite{LMR} where it is introduced a foliated version of the Ricci flow. 
Subsequently,  Smoczyk, Wang and Zhang proved in \cite{smoczyk} that the transverse Ricci flow preserves the Sasakian condition and study its long time behavior generalising the work of Cao in \cite{cao} to the Sasakian case. Some deep geometric and analytic aspects of the Sasaki Ricci flow were further investigated in \cite{collins0,collins1,collins2,collins3}. 

In analogy to the K\"ahler case, the {\em Sasaki $J$-flow} introduced in this paper (see Section \ref{wellposed} for the precise definition) is the gradient flow of a functional $J_{\chi}\colon \mathcal{H}\to \R$ whose definition depends on the choice of a transverse K\"ahler structure $\chi$. Sasakian metrics arising from critical points of the restriction of  $J_{\chi}$ to the space of {\em normalized} Sasakian potentials $\mathcal{H}_0$, are natural candidates to be {\em canonical} Sasakian metrics.



%
The main result of the paper is the following 

\begin{theorem}\label{main} Let $(M,\xi,\Phi,\eta,g)$ be a $(2n+1)$-dimensional Sasakian  manifold and let $\chi$ be a transverse K\"ahler form on $M$. Then the functional $J_{\chi}\colon \mathcal H_0\to \R$ 
has at most one critical point and the Sasaki $J$-flow has a long-time solution $f$ for every initial datum $f_0$.  Furthermore, the length of any smooth curve in $\mathcal{H}_0$ and the distance between any two points decrease under the flow and when the transverse holomorphic bisectional curvature of $\chi$ is nonnegative, $f$ converges to a critical point of $J_{\chi}$ in $\mathcal H_0$.
\end{theorem}
The last sentence in the statement of Theorem \ref{main} implies that  if the transverse K\"ahler structure $\chi$ has nonnegative  transverse holomorphic bisectional curvature, then $J_{\chi}$ has a critical point in $\mathcal H_0$. We remark that Sasakian manifolds having nonnegative  transverse holomorphic bisectional curvature are classified in \cite{HeSun}, but in the definition of the Sasaki $J$-flow, $\chi$ is just a transverse K\"ahler structure not necessarily induced by a Sasaki metric. 

From the local point of view, a  solution to the {\em Sasaki $J$-flow} can be seen as a collection of
solutions to the K\"ahler $J$-flow on open sets in $\C^n$. This fact allows us to use all the local estimates about the K\"ahler $J$-flow provided in \cite{chen}. What is necessary modifying from the K\"ahler case is the proof of the existence of a short-time solution to the flow (since the flow is parabolic only along transverse directions) and the 
 global estimates. The short-time existence is obtained in Section \ref{wellposed} by using a trick introduced in \cite{smoczyk}, while the global estimates are obtained by using a transverse version of the maximal principle for transversally elliptic operators (see section \ref{maxprin}). 

\medskip
\noindent {\em{Acknowledgements}}. The authors would like to thank Valentino Tosatti for useful comments and remarks. 

\section{Preliminaries}\label{preliminaries}
In this section we recall some basic facts about Sasakian Geometry declaring the notation which will be adopted in the rest of the paper.
  
Let $(M,\xi,\Phi,\eta,g)$ be a $(2n+1)$-dimensional Sasaki manifold. Then the Reeb vector field $\xi$ specifies a Riemannian foliation on $M$, which is usually denoted by $\mathcal F_{\xi}$, and the vector bundle to $M$ splits in  $TM=\mathcal D\oplus L_{\xi}$, where $L_{\xi}$ is the line bundle generated by $\xi$ and  $\mathcal D$ has as fiber over a point $x$ the vector space $\ker \eta_x$.  The metric $g$ splits accordingly in $g=g^T+\eta^2$, where the degenerate tensor $g^T$ is called the {\em transverse metric} of the Sasakian structure. In the following we denote by $\nabla^T$ the {\em transverse Levi-Civita} connection defined on the bundle $\mathcal D$ in terms of the Levi-Civita connection ${\rm D}$ of $g$ as   
\b\label{nablaT}
\nabla^T_XY=
\begin{cases}
{\rm D}_XY\quad\mbox{ if } X\in \Gamma(\mathcal D)\\
[\xi,Y]^{\mathcal D} \quad\mbox{ if } X=\xi\,,
\end{cases}
\e
where the upperscript $\mathcal D$ denotes the orthogonal projection onto $\mathcal D$. This connection induces the transverse curvature 
\b\label{RT}
R^T(X,Y)Z=\nabla_X^T\nabla^T_YZ-\nabla_Y^T\nabla^T_XZ-\nabla^T_{[X,Y]}Z,
\e
and the transverse Ricci curvature ${\rm Ric}^T$ obtained as the trace of the map $X\mapsto R^T(X,\cdot)\cdot$ on $\mathcal D$ with respect to $g^T$. We further recall that a real $p$-form $\alpha$ on $(M,\xi,\Phi,\eta,g)$ is called {\em basic} if 
$$
\iota_{\xi}\alpha=0\,,\quad \iota_{\xi}d\alpha=0,
$$
where $\iota_{\xi}$ denotes the contraction along $\xi$. The set of basic $p$-forms is usually denoted by $\Omega_{B}^p(M)$ and $\Omega_B^0(M)=C^{\infty}_B(M)$. Since the exterior differential operator takes basic forms into basic forms, its restriction $d_B$ to $\Omega_B(M)=\oplus \Omega_B^p(M)$ defines a cohomological complex. Moreover, $\Phi$ induces a {\em transverse complex structure} $J$ on $(M,\xi)$ and a splitting of the space of complex basic forms in forms of type $(p,q)$ in the usual way. Furthermore, the complex extension of $d_B$ to $\Omega_B(M,\C)$  splits as $d_B=\partial_B+\bar\partial_B$ and $\bar\partial_B^2=0$
(see e.g. \cite{bg} for details). A basic $(1,1)$-form $\chi$ on  $(M,\xi,\Phi,\eta,g)$ is said to be {\em positive} if 
\b\label{positive}
\chi(Z,\bar Z)>0,
\e 
for every non-zero section $Z$ of $\Gamma(\mathcal D^{1,0})$.  If further $\chi$ is closed, we refer to $\chi$ as to a {\em transverse K\"ahler form}. Note that condition \eqref{positive} depends only on the transverse complex structure $J$ and  on $\xi$, since $\chi$ is basic. 
Every such a $\chi$ induces the global metric 
$$
g_{\chi}(\cdot,\cdot)=\chi(\cdot,\Phi\cdot)+\eta(\cdot)\,\eta(\cdot),
$$
on $M$. The metric $g_{\chi}$ induces a transverse Levi-Civita connection $\nabla^{\chi}$ and a tranverse curvature $R^{\chi}$ as in \eqref{nablaT} and \eqref{RT} (here it is important that $\chi$ is basic in order to define $\nabla^{\chi}$).  
\subsection{Adapted coordinates}\label{foliatedcoordinates}
%
%
%
Let $(M,\xi,\Phi,\eta,g)$ be a Sasakian manifold.  We can always find  local coordinates
$\{z^1,\dots,z^{n},z\}$ taking values in $\C^n\times \R$ such that 
\b\label{coord}
\xi=\partial_z\,,\quad 
\Phi(d{z^j})=i\,d{z^j},\quad \Phi(d{\bar z^j})=-i\,d{\bar z^j}\,.
\e

A function $h$ is basic if and only if does not depend on the variable $z$ and we usually denote by 
$h_{,i_1\dots i_r\bar j_1\dots \bar j_l}$ the space derivatives of $h$ along $\partial_{z^{i_1}},\dots ,\partial_{z^{i_r}},
\partial_{\bar z^{ j_1}},\dots,\partial_{\bar z^{ j_l}}$. We denote by $A_{i_1\dots i_r\bar j_1\dots \bar j_l}$ (without \lq\lq ,\rq\rq) the components of the basic tensor $A$. Furthermore, when a function $f$ depends also on a time variable $t$, we use notation 
$\dot f$ to denote its time derivative.  In the case when $f$ depends on two time variables $(t,s)$, we write $\partial_t f$ and $\partial_s f$, to distinguish the two derivatives. 
 
For instance, the metric $g$  and the transverse symplectic form $d\eta$ locally write as 
$$
g=g_{i\bar j}dz^id\bar z^{j}+\eta^2\,,\quad d\eta=2ig_{i\bar j}dz^i\wedge d\bar z^{j},
$$
where the $g_{i\bar j}$ are all basic functions. In particular the transverse metric $g^T$ writes as 
$g^T=g_{i\bar j}dz^id\bar z^{ j}$ and a Sasakian structure can be regarded as a collection of K\"ahler structures each one defined on an open set of $\C^n$.  Observe that conditions \eqref{coord} depend only $(\xi,J)$ and therefore they hold for every Sasakian structure compatible with $(\xi,J)$. 
This fact is crucial in the proof of Theorem \ref{main}.  
%

In this paper we make sometimes use of {\em special foliated coordinates} with respect to a transverse K\"ahler form $\chi$. Indeed, once a transverse K\"ahler form $\chi$ on the Sasakian manifold $(M,\xi,\Phi,\eta,g)$ is fixed, we can always find foliated coordinates $\{z^1,\dots,z^{n},z\}$ around any fixed point $x$ such that if $\chi=\chi_{i\bar j}\,dz^i\wedge d\bar z^j$,  then 
$$
\chi_{i\bar j}=\delta_{ij}\,,\quad \partial_{z^r}\chi_{i\bar j}=0\,,\mbox{ at }x\,. 
$$
Moreover, we can further require that the transverse metric $g^T$ takes a diagonal expression at $x$. 
%
%

\subsection{The space of the Sasakian potentials and the definition of $J$-flow} Following \cite{bgs,guanzhang,guanzhangA,he}, given a Sasakian manifold $(M,\xi,\Phi,\eta,g)$, we consider 
$$
{\mathcal{H}}=\{ h \in C_B^{\infty}(M,\R)\,\,:\,\, \eta_h=\eta+ d^c h\ \mbox{is a contact form}\}\,,
$$
where $d^ch$ is the $1$-form on $M$ defined by $(d^ch)(X)=-\frac12 dh(\Phi(X))$. Every $h\in {\mathcal{H}}$ induces the Sasakian structure
$(\xi,\Phi_h,\eta_h,g_h)$ where 
$$\begin{aligned}
\Phi_h=\Phi-(\xi\otimes (\eta_h-\eta))\circ \Phi\,,\quad 
g_h=\frac12\,d\eta_h\circ({\rm Id}\otimes \Phi_h)+\eta_h\otimes\eta_h\,.
\end{aligned}
$$
Notice that 
$$
\eta_h\wedge (d\eta_h)^n=\eta \wedge (d\eta_h)^n\,.
$$ 
%
%
%
All the Sasakian structures induced by  the functions in $\mathcal H$ have the same Reeb vector field and the same transverse complex structure. It is rather natural to restrict our attention to the space of $\mathcal{H}_0$ of {\em normalized } Sasakian potentials. $\mathcal H_0$ is defined as the zero set of the functional  
 $I\colon \mathcal H\to \R$ defined trough its first variation by 
$$
\frac{\partial}{\partial t}I(f)=\frac{1}{2^nn!}\int_M \dot{f}\,\eta\wedge d\eta_{f}^n\,,\quad I(0)=0,
$$
where $f$ is a smooth curve in $\mathcal H$
(see \cite[formula (14)]{guanzhang} for an explicit formulation of $I$).
 The pair $(\xi,J)$ can be seen as a {\em polarisation} of the Sasakian manifold (see \cite{bgs}). Notice that $\mathcal{H}$ is open in  $C^{\infty}_B(M,\R)$ and has the natural Riemannian metric
\begin{equation}\label{metric}
(\varphi,\psi)_h:=\frac{1}{2^nn!}\int_M\, \varphi \psi\,\eta \wedge (d\eta_h)^n\,.
\end{equation}
The covariant derivative of \eqref{metric} along a smooth curve $f=f(t)$ in $C^{\infty}_B(M,\R)$ takes the following expression    
$$
D_t\psi=\dot\psi-\frac14\,\langle d_B\psi,d_B\dot f\rangle_f,
$$
where $\psi$ is an arbitrary smooth curve in $C^{\infty}_B(M,\R)$ and $\langle \cdot ,\cdot\rangle _f$ is the pointwise scalar product induced by $g_f$ on basic forms (see \cite{guanzhang,he}). Note that $D_t$ can be alternatively written as
$$
D_t\psi=\dot\psi-\frac12{\rm Re} \langle \partial_B\psi,\partial_B\dot f\rangle_f
$$
which has the following local expression
$$
D_t\psi=\dot\psi-\frac14g_f^{\bar jk}(\psi_{,k}\dot f_{,\bar j}+\psi_{,\bar j}\dot f_{,k})\,.
$$
Moreover, a curve $f=f(t)$ in $\mathcal H$ is a geodesic if and only if it solves
 \begin{equation}\label{geodesic}
\ddot f-\frac14 |d_B \dot f|^2_f=0\,.
\end{equation}
%
Furthermore, W. He proved in \cite{he} that $\mathcal H$ is an infinite dimensional symmetric space whose curvature can be written as
$$
R_h(\psi_1,\psi_2)\psi_3=-\frac1{16}\{\{\psi_1,\psi_2\}_f,\psi_3\}_h,
$$
where $\{\,,\,\}_h$ is the Poisson bracket on $C^{\infty}_B(M,\R)$ induced by the contact form $\eta_h$. 

As in the K\"ahler case, it is still an open problem to establish when two points in $\mathcal H$ can be connected by a geodesic path. Fortunately, Guan and Zhang proved in \cite{guanzhangA} that this can be always  done in a weak sense. More precisely, the role of $\mathcal H$ 
is replaced with its completion $\bar{\mathcal{H}}$ with respect to the $C_{w}^2$-norm (see \cite{guanzhangA} for details) and the geodesic equation \eqref{geodesic} with 
\b\label{wgeodesic}
\left(\ddot f-\frac14 |d_B \dot f|^2_f\right)\,\eta\wedge d\eta_f^n=\epsilon\,\eta\wedge d\eta^n\,.
\e
Then, by definition a $C^{1,1}$-geodesic is a curve in $\bar {\mathcal H}$ obtained as weak limit of solutions to \eqref{wgeodesic}, and from \cite{guanzhangA} it follows that for every 
two points in $\mathcal H$ there exists a $C^{1,1}$-geodesic connecting them. 

\medskip 
%
%


Now we can introduce the Sasakian version of the $J$-flow.  The definition depends on the choice of a transverse K\"ahler form $\chi$. Note that  
$$
 \eta_h\wedge \chi^n=\eta\wedge \chi^n\neq 0,
$$
for every $h\in \ \mathcal{H}$, since $\chi$ and $d^c_Bh$ are both basic forms. 

\begin{prop}\label{wp}
Let $f_0,f_1\in  \mathcal{H}$ and $f\colon [0,1]\to\mathcal{H}
$ be a smooth path satisfying $f(0)=f_0$, $f(1)=f_1$. Then  
$$
A_{\chi}(f):=\int_0^1\int_M \dot f\,\chi\wedge \eta\wedge (d\eta_f)^{n-1}\,dt,
$$
depends only on $f_0$ and $f_1$.
\end{prop}
\begin{proof}
Following the approach of Mabuchi in \cite{mabuchiF},
let $\psi(s,t):=sf(t)$ and let $\Psi$ be the $2$-form on the square  $Q=[0,1]\times [0,1]$ defined as 
$$
\Psi(s,t):=\left(\int_M \partial_t\psi\, \eta\wedge\chi\wedge (d\eta_\psi)^{n-1}\right)\,dt+ 
\left(\int_M  \partial_s\psi\, \eta\wedge\chi\wedge (d\eta_\psi)^{n-1}\right)\,ds\,.
$$
We show that $\Psi$ is closed as $2$-form on $Q$:
\begin{equation}
\begin{split}
d\Psi(s,t)=&\,\frac{d}{ds}\left( \int_M \partial_t\psi\, \eta\wedge\chi\wedge (d\eta_\psi)^{n-1}\right)\,dt\wedge ds-\frac{d}{dt}\left(\int_M  \partial_s\psi\, \eta\wedge\chi\wedge (d\eta_\psi)^{n-1}\right)\,dt\wedge ds\\
=&\,(n-1)\,s\,i\left( \int_M \dot f\, \eta\wedge\chi\wedge(d\eta_\psi)^{n-2}\wedge \partial_B \bar\partial_B f+\int_M  f\, \eta\wedge\chi\wedge (d\eta_\psi)^{n-2}\wedge\bar \partial_B \partial_B \dot f\right)dt\wedge ds\\
=&\,(n-1)\,s\,i\left[ \int_M d\left(\dot f\, \eta\wedge\chi\wedge(d\eta_\psi)^{n-2}\wedge  \bar\partial_B f\right)-\int_M\partial_B\dot f\, \eta\wedge\chi\wedge(d\eta_\psi)^{n-2}\wedge  \bar\partial_B f+\right.\\
&\left.+\int_M  d\left( f\, \eta\wedge\chi\wedge (d\eta_\psi)^{n-2}\wedge \partial_B \dot f\right)-\int_M  \bar \partial_B f\wedge \eta\wedge\chi\wedge (d\eta_\psi)^{n-2}\wedge \partial_B \dot f\right]dt\wedge ds\\
=&\,(n-1)\,s\,i\left[ \int_M\partial_B\dot f\wedge  \bar\partial_B f \wedge\eta\wedge\chi\wedge(d\eta_\psi)^{n-2} -\int_M  \partial_B \dot f\wedge \bar \partial_B f\wedge \eta\wedge\chi\wedge (d\eta_\psi)^{n-2} \right]dt\wedge ds\\
=&\,0.
\end{split}\nonumber
\end{equation}
Therefore the Gauss-Green Theorem implies that
$$
\int_{\partial Q}\Psi=0,
$$
and the claim follows.
\end{proof}

In view of the last proposition, we can write $A_{\chi}(f_0,f_1)$ instead of $A_{\chi}(f)$. 
\begin{definition}
The {\em Sasaki $J$-functional} is the map $J_{\chi}\colon \mathcal{H}\to \R$ defined as 
$$
J_{\chi}(h)=\frac{1}{2^{n-1}(n-1)!} A_{\chi}(0,h)\,. 
$$
\end{definition}
Alternatively we can define $J_{\chi}$ through its first variation by
\b\label{ptJ} 
\partial_t J_{\chi}(f)=\int_M \frac{1}{2^{n-1}(n-1)!}\dot f\,\chi\wedge \eta\wedge (d\eta_f)^{n-1}\,,\quad J_{\chi}(0)=0,
\e
and then apply Proposition \ref{wp} to show that the definition is well-posed. 
 Note that 
$$
\partial_t J_{\chi}(f)=\frac12 (\dot f\chi,d\eta)_f,
$$
and therefore
$$
\partial_t J_{\chi}(f)=\frac{1}{2^nn!} \int_M \dot f\sigma_f\,\eta\wedge d\eta_{f}^n,
$$
where 
for $h\in \mathcal {H}$
$$
\sigma_h=g^{\bar ba}_h\chi_{a\bar b} ,
$$
the components and the derivatives are computed with respect to transverse holomorphic coordinates and with the upper indices in $g_h$   we denote the components of the inverse matrix.

If we restrict $J_{\chi}$ to $\mathcal{H}_0,$ then $h\in \mathcal H_0$ is a critical point of $J_{\chi}\colon \mathcal{H}_0\to \R$
if and only if 
$$
\int_M k\,\eta\wedge \chi\wedge d\eta_h^{n-1}=0,
$$
for  every $k$ in  the tangent space to $\mathcal H_0$ at $h$, i.e. if and only if $2n \,\eta\wedge \chi\wedge d\eta_h^{n-1}=c\, \eta\wedge d\eta_h^{n}$, where 
\b\label{c}
c=\frac{2n\int_M \chi \wedge \eta\wedge d\eta^{n-1}}{\int_M \eta\wedge d\eta^n}\,. 
\e 
Given  $h\in\mathcal{H}_0$,  we can rewrite the condition of being a critical point of $J_{\chi}$ as 
 \b\label{criticaleq}
\sigma_h=c\,.
\e
Therefore,  if $f_0\in \mathcal H_0$ is fixed, the evolution equation 
\b\label{jflow}
\dot f=c-\sigma_f\,,\quad f(0)=f_0,
\e
can be seen as the gradient flow of $J_{\chi}\colon \mathcal{H}_0\to \R$.

\begin{definition}
A Sasakian structure $(\xi,\Phi_h,\eta_h,g_h)$ is called {\em critical} if $h$ satisfies \eqref{criticaleq}. We will refer to \eqref{jflow}
 as to the {\em Sasaki $J$-flow}. 
 \end{definition}
 
\section{Technical Results and Critical Sasaki Metrics}
Let $(M,\xi,\Phi,\eta,g)$ be a $(2n+1)$-dimensional compact Sasakian manifold and let 
$f=f(t)$ be a smooth curve in the space of normalized Sasakian potentials $\mathcal H_0$. 
Then  
\begin{equation}\label{vol}
\frac{\partial}{\partial t} \eta \wedge (d\eta_f)^n= \Delta_f\dot f\, \eta \wedge (d\eta_f)^n,
\end{equation}
where  for $h\in \mathcal H_0$, $\Delta_h$  denotes the basic Laplacian 
%
$$
\Delta_h\psi=-\partial_B^*\partial_B\psi=\,g_h^{\bar jr}\psi_{,r\bar j}\,,\quad \mbox{ for }\psi\in C^{\infty}_B(M, \R)\,.
$$
A direct computation yields  
\begin{equation}\label{sigma_t}
\dot \sigma_f=-g_f^{\bar p m}\, \dot f_{,m\bar l}\,g_f^{\bar l q}\,\chi_{q\bar p}=-\langle i\partial_B\bar\partial_B \dot f,\chi\rangle_f,
\end{equation}
where, given $\alpha$ and $\beta$ in $\Omega_B^{(p,q)}(M,\C)$, we set 
$$
\langle \alpha,\beta\rangle_h=\alpha_{i_1\dots i_p\bar j_1\dots\bar j_q}\cdot \bar \beta_{r_1\dots r_p\bar s_1\dots\bar s_q} 
g^{\bar r_1i_1}_{h}\cdots g^{\bar r_pi_p}_{h}\cdot g_h^{\bar j_1 s_1}\cdots g_{h}^{\bar j_q s_q},
$$
and 
$$
(\alpha,\beta)_h=\frac{1}{2^nn!}\int_M \langle \alpha,\beta\rangle_h\,\eta\wedge d\eta_{h}^n\,. 
$$
In particular, if $\alpha=\alpha_i\,dz^i$ and $\beta=\beta_j\,dz^j$ are transverse forms of type $(1,0)$, 
 by writing $\chi=i\chi_{a\bar b}dz^a\wedge \bar z^b$, we have 
$$
\langle\chi,\alpha\wedge \bar \beta\rangle_h=i\chi_{a\bar b}\bar \alpha_{r}\beta_j g^{\bar r a}_hg^{\bar b j}_h\,.
$$
\smallskip

The following technical lemma will be useful in the sequel.
\begin{lemma}\label{int1}
Let $u\in C_B^\infty (M,\mathds{R})$ and $f$ be a smooth path in $C^{\infty}_B(M,\R)$. Then
\begin{itemize}
\item[($i$)] $(\Delta_f\dot f,u\sigma)_f=-(\partial_B \dot f,\sigma\partial_B u)_f-(u\partial_B\dot f,\partial_B \sigma)_f;$

\vspace{0.1cm}
\item[($ii$)] $(\bar \partial_B\partial_B \dot f,u\chi)_f=-i(u\,\partial_B \dot f,\partial_B\sigma)_f-(\chi, \partial_B u  \wedge\bar\partial_B  \dot f)_f;$

\vspace{0.1cm}
\item[($iii$)] $(\dot f,\dot \sigma)_f=\frac12 (\partial_B (\dot f)^2,\partial_B\sigma)_f-i(\chi, \partial_B \dot f  \wedge\bar\partial_B  \dot f)_f. $
\end{itemize}
where $\sigma=g_f^{\bar kr}\chi_{r\bar k}$. 
\end{lemma}
\begin{proof}$\textrm{}$
\begin{itemize}
\item[($i$)]
$
(\Delta_f\dot f,\dot u\sigma)_f=-(\partial_B^*\partial_B\dot f,u \sigma)_f=-(\partial_B \dot f,\sigma\partial_B u)_f-(u\partial_B\dot f,\partial_B\sigma)_f.
$

\vspace{0.1cm}
\item[($ii$)] Since the Laplacian is self-adjoint we have:
\begin{equation}
\begin{split}
2^nn!i(\bar \partial_B\partial_B \dot f,u\chi)_f=&-\int_M ug_f^{\bar c j}g_f^{\bar b a}\dot f_{,j\bar b}\,\chi_{a\bar c}\,\eta\wedge(d\eta_f)^n\\
=&\int_M u_{,\bar b}g_f^{\bar c j}g_f^{\bar b a}\chi_{a\bar c}\,\dot f_{,j}\,\eta\wedge(d\eta_f)^n+\int_M ug_f^{\bar c j}g_f^{\bar b a}\chi_{a{\bar b},\bar c}\,\dot f_{,j}\,\eta\wedge(d\eta_f)^n\\
=&\int_M u_{,\bar b}g_f^{\bar c j}g_f^{\bar b a}\chi_{a\bar c}\,\dot f_{,j}\,\eta\wedge(d\eta_f)^n+\int_M ug_f^{\bar c j}\sigma_{,\bar c}\,\dot f_{,j}\,\eta\wedge(d\eta_f)^n\\
=&2^nn!(u\partial_B \dot f,\partial_B\sigma)_f-2^nn!i(\chi, \partial_B u\wedge\bar \partial_B \dot f)_f.
\end{split}\nonumber
\end{equation}

\vspace{0.1cm}
\item[($iii$)] By using \eqref{sigma_t} and $(ii)$, we have 
$$
\begin{aligned}
(\dot f,\dot \sigma)_f=&-i(\partial_B\bar \partial_B\dot f,\dot f\chi)_f=(\dot f\,\partial_B \dot f,\partial_B\sigma)_f-i(\chi, \partial_B \dot f  \wedge\bar\partial_B  \dot f)_f\\
=&\frac12 (\partial_B (\dot f)^2,\partial_B\sigma)_f-i(\chi, \partial_B \dot f  \wedge\bar\partial_B  \dot f)_f
\end{aligned}
$$

\end{itemize}
as required. 
\end{proof}

 The following proposition is about the uniqueness of {\em critical Sasaki metrics} in $\mathcal H_0$ and it is analogue to the K\"ahler case.  
\begin{prop}\label{stimevarie}
$J_{\chi}\colon \mathcal{H}_0\to \R$ has at most one critical point. 
\end{prop}
\begin{proof}
Let $f$ be a curve in the space $\bar{\mathcal H}$ obtained as completion of $\mathcal H$ with respect to the $C^2_w$-norm.  Then taking into account the definition of $J_{\chi}$, 
Lemma \ref{int1} and equations \eqref{vol}, \eqref{sigma_t}, we have 
\begin{equation}
\begin{split}
\partial_t^2\,J_{\chi}(f)=&(\ddot f,\sigma_f)_f+\frac12(\Delta_f\dot f,\dot f\sigma_f)_f+i(\dot f\bar\partial_B\partial_B\dot f,\chi)_f\\
=&\frac{1}{2^nn!}\int_M\left(\ddot f-\frac12|\partial_B \dot f|_g^2 \right)\sigma_f\,\eta\wedge (d\eta_f)^{n}-i(\chi,\partial_B \dot f\wedge \bar\partial_B\dot f)_f\,. 
\end{split}\nonumber
\end{equation}
Therefore if $f$ solves the modified geodesic equation \eqref{wgeodesic}, then 
$$
\partial_t^2\,J_{\chi}(f)\geq -i(\chi,\partial_B \dot f\wedge \bar\partial_B\dot f)_f\geq 0\,. 
$$
Let us assume now to have two critical points $f_0$ and $f_1$ of $J_{\chi}$ in $\mathcal H_0$ and denote by $\bar{\mathcal H}_0$
the competition of $\mathcal H_0$ with respect to the $C^2_w$-norm. 
Then, in view of \cite{guanzhangA}, there exists a $C^{1,1}$-gedesic $f$ in $\bar{\mathcal H}_0$  such that  $f(0)=f_0$ and $f(1)=f_1$. Let $h(t)=J_{\chi}(f(t))$. Then since $f_0$ and $f_1$ are critical points of $J_{\chi}$, we have $\dot h(0)=\dot h(1)=0$. Since $\ddot h\geq 0$, it as to be $\ddot h\equiv 0$ which implies $\partial_B \dot f=0$ and $\dot f(t)$ is constant for every $t\in [0,1]$. Finally, since $f$ is a curve in $\bar{\mathcal H}_0$, then $I(f)=0$ and therefore $\dot f=0$, which implies $f_0=f_1$, as required.   
\end{proof}

On a compact $3$-dimensional Sasaki manifold, the existence of a critical metric is always guaranteed. Indeed,   
if $(M,\xi,\Phi,\eta,g)$ is a compact $3$-dimensional Sasaki manifold with a fixed background transverse K\"ahler form $\chi$, then we can write:
$$
\chi=\frac14\langle \chi ,d\eta\rangle\,d\eta\,,\quad  d\eta_h= \left(1-\frac12 \Delta_Bh\right)\, d\eta,
$$
where the scalar product and the basic Laplacian are computed with respect to the metric induced by $\eta$. Hence, $\eta_h=\eta+d^ch$ induces a critical metric if and only if $h$ solves:
$$\Delta_Bh=2-\frac{1}{c}\langle \chi ,d\eta\rangle\,,\quad  c=\frac{2\int_M\eta\wedge\chi}{\int_M\eta\wedge d\eta}\,,$$
which has always a solution since:  
$$
\int_M \left(2-\frac{1}{c}\langle \chi ,d\eta\rangle\right)\eta\wedge d\eta=0\,. 
$$

In higher dimensions there is a cohomological obstruction  to the existence of a critical metric  similar to the one in the K\"ahler case. 

Recall that if $(M,\omega)$ is a compact K\"ahler $2n$-dimensional manifold (with $n>1$) with a fixed background  K\"ahler metric $\chi$, then the existence of a $J_{\chi}$-critical normalised K\"ahler potential on $(M,\omega)$
implies that $[c\omega-\chi]$ is a K\"ahler class in $H^2(M,\R)$ (see \cite{donaldsonMM}).  

In \cite{chenM} Chen proved that such a condition is sufficient for the existence of a critical metric on complex surfaces, while in the recent paper \cite{LSS}, Lejmi and Sz\'ekelhyidi provide an example where it is satisfied, but the $J$--flow does not converge. In \cite{songweinkoveJ}, Song and Weinkove find a necessary and sufficient condition for the convergence of the flow in terms of a $(n-1,n-1)$-form. Some further results have been obtained in \cite{wein1,wein2}.


 The Sasakian context is quite similar. Indeed, given a Sasakian manifold $(M,\xi,\Phi,\eta,g)$ with a fixed background transverse K\"ahler form $\chi$, then if $h\in \mathcal H_0$ is a critical normalised Sasakian potential, then 
$\frac c2d\eta_h-\chi$ is a transverse K\"ahler form and we expect that the results in \cite{songweinkoveJ,wein1,wein2} could be generalised to the Sasakian case.

The following proposition is about the existence of a critical Sasaki metric in dimension $5$: 

\begin{prop}
Let  $(M,\xi,\Phi,\eta,g)$ be a compact $5$-dimensional Sasaki manifold. Assume that there exists a map $h\in \mathcal{H}_0$ such that  $\frac c2\,(d\eta+dd^ch)-\chi$ is a transverse K\"ahler form. Then, there exists a critical Sasaki metric on $M$. 
\end{prop}
\begin{proof}

Up to rescaling $\eta$, we may assume $c=1$.  A function $h\in \mathcal{H}_0$ is critical if and only if  
$$
2\,\eta\wedge \chi\wedge\left(\frac12d \eta+dd^c h\right)= \eta\wedge \left(\frac12d\eta+dd^ch\right)^2\,. 
$$
Let $\Omega=\frac12 d\eta-\chi$. Then our hypothesis implies that $\Omega$ is a transverse K\"ahler form and moreover 
by substituting we get 
$$
(\Omega+dd^c h)^2=\chi^2\,. 
$$
Finally, the Calabi-Yau theorem in K\"ahler foliations \cite{aziz} implies the statement. 
\end{proof}

\section{Well-posedness of the Sasaki $J$-flow}  \label{wellposed}

\begin{theorem}
The Sasaki $J$-flow is well-posed, i.e., for every initial datum $f_0$, system \eqref{jflow} has a unique maximal solution $f$ defined in $[0,\epsilon_{\max})$, for some positive $\epsilon_{\max}$. 
\end{theorem}
\begin{proof}
Since $\mathcal{H}$ is not open in $C^{\infty}(M,\R)$, to apply the standard parabolic theory we have to use a trick adopted by Smoczyk, Wang and Zhang for showing the short-time existence of the Sasaki-Ricci flow in \cite{smoczyk}. 
Since the functional $F\colon \mathcal{H}\to \R$ defined as 
$$
F(f)=\xi^2(f)+\sigma_f,
$$
is elliptic, the standard parabolic theory implies that the geometric flow 
\begin{equation}\label{modifiedflow}
\dot f=c-\xi^2(f)-\sigma_f,\quad 
f(0)=f_0,
\end{equation}
has a unique maximal solution $f\in C^{\infty}(M\times[0,\epsilon_{\rm max}),\R)$, for some $\epsilon_{\rm max}>0$. Of course if $f(\cdot,t)$ is a solution to \eqref{modifiedflow} which is basic for every $t$ and $I(f)=0$, then $f$ solves \eqref{jflow}. We first show that when $f_0$ is basic, then the solution $f$ to \eqref{modifiedflow} holds basic for every $t\in[0,\epsilon_{\rm max})$. We have 
$$
\partial_t \xi(f)=\xi(\dot f)=\xi(-\xi^2(f)-g^{\bar k r}\chi_{r\bar k})
$$
Moreover, since the components of $\chi$ are basic, we have 
$$
\xi(g^{\bar k r}\chi_{r\bar k})=-g^{\bar k l}_f(\xi(f_{,l\bar m}))g^{\bar m r}_f\chi_{r\bar k}=
g^{\bar k l}_f\xi (f)_{,l\bar m}\,g^{\bar m r}_f\chi_{r\bar k}=-\langle dd^c_B\xi (f),\chi \rangle_f\,,
$$
i.e. 
\b\label{xi}
\partial_t \xi(f)=-\xi^3(f)+\langle dd^c_B\xi (f),\chi \rangle_f\,. 
\e
Equation \eqref{xi} is parabolic in $\xi(f)$ and then, since the solution to  a parabolic problem is unique, if $\xi(f_0)=0$, $\xi(f(t))=0$ for every $t\in [0,\epsilon_{\rm max})$, as required.  
Finally we show that if $f_0$ is normalised, then $I(f)=0$ for every $t\in [0,\epsilon_{\rm max})$.  We have 
$$
\partial_tI(f)=\frac{1}{2^nn!}\int_M \dot f\eta\wedge d\eta_{f}^n=\frac{1}{2^nn!}\int_M (c-\sigma_f)\,\eta\wedge d\eta_{f}^n,
$$
and since $c\int_M\,\eta\wedge d\eta_{f}^n=\int_M\sigma_f\,\eta\wedge d\eta_{f}^n$ we have $\partial_tI(f)=0$. Therefore, since $I(f_0)=0$, $I(f)=0$ for every $t\in[0,\epsilon_{\max})$ and the claim follows. 
\end{proof}

\begin{rem}{\em 
Alternatively, the short-time existence of the Sasaki $J$-flow can be obtained by invoking the short-time existence of any second order {\em transversally parabolic equation} on compact manifolds foliated by Riemannian foliations. A proof of the latter result can be founded in \cite{BHV}.  }
\end{rem}

In analogy to the K\"ahler case, let ${\rm En} \colon \mathcal H_0\to \R$ be the energy functional 
$$
{\rm En}(h)=\frac{1}{2^nn!}\int_M\sigma_h^2\,\eta\wedge (d\eta_h)^{n}=(\sigma_h,\sigma_h)_{h}^2.
$$

\begin{prop}\label{stimevarie}
The following items hold:
\begin{itemize}
\item[1.] ${\rm En}$ has the same critical points of $J_\chi$ and it is strictly decreasing along the Sasaki $J$-flow;
\item[2.] any critical point of ${\rm En}$ is a local minimizer;
\item[3.] the lenght of any curve in $\mathcal H_0$ and the distance of any two points in $\mathcal H_0$ decrease under the $J$-flow.
\end{itemize}
\end{prop}
\begin{proof}
%
$1.$ Let $f\!:[0,1]\to \mathcal{H}_0$ be a smooth curve. Then, by using \eqref{sigma_t} and Lemma \ref{int1}, the first variation of $E$ reads:
\begin{equation*}
\begin{split}
\partial_t {\rm En}(f)=&\frac{1}{2^nn!}\partial_t\int_M\sigma_f^2\,\eta\wedge (d\eta_f)^{n}
=2(\sigma_f,\dot \sigma_f)_f+(\sigma_f^2,\Delta_f\dot f)_f\\
=&2(\sigma_f\partial_B \dot f,\partial_B\sigma_f)-2i(\chi,\partial_B\dot f\wedge\bar \partial_B \sigma_f)_f-2(\partial_B\sigma_f,\sigma_f\partial_B\dot f)_f\\
=&-2i(\chi,\partial_B\dot f\wedge\bar \partial_B \sigma_f)_f.
\end{split}
\end{equation*}
Along the Sasaki $J$-flow one has $\dot f=c-\sigma_f$, thus:
$$
\partial_t {\rm En}(f)=-2i(\chi,\partial_B\sigma_f\wedge\bar \partial_B \sigma_f)_f\leq 0,
$$
and ${\rm En}$ is strictly decreasing along the $J$-flow. Moreover, if $h\in \mathcal H_0$ is a critical point of ${\rm En}$, then 
$\partial_B \sigma_h=0$ which implies that $h$ is critical if and only if $\sigma_h=c$. 

\medskip  
$2.$
Now we compute the second variation of ${\rm En}$. Let $f\colon (-\delta,\delta) \times (-\delta,\delta) \to \mathcal H_0$ be a smooth map in the variables $(t,s)$ 
 Assume that $f(0,0)=h$ is a critical point of ${\rm En}$ and let $u=\partial_{t}f_{|(0,0)}$, $v=\partial_s f_{|(0,0)}$. 
Then we have  
$$
\partial_s\partial_t {\rm En}(\alpha)=\frac{1}{2^{n-1}n!}i\partial_s\left(\int_M\langle\chi,\partial_B\partial_t\alpha\wedge\bar \partial_B \sigma_\alpha\rangle_\alpha\,\eta\wedge(d\eta_\alpha)^n\right)\,,
$$
and 
$$
\partial_s\partial_t {\rm En}(\alpha)_{|(0,0)}=\frac{1}{2^{n-1}n!}\int_M\langle\chi,\partial_Bu\wedge\bar \partial_B \partial_s\sigma_{\alpha |(0,0)}\rangle_h\,\eta\wedge(d\eta_h)^n=2(\chi,\partial_Bu\wedge\bar \partial_B \partial_s\sigma_{\alpha |(0,0)})_h,
$$
since $\sigma_h$ is constant. Now 
$$
2(\chi,\partial_Bu\wedge\bar \partial_B \partial_s\sigma_{\alpha |(0,0)})_h=2 (\chi,\partial_s\sigma_{\alpha |(0,0)}\, \partial_B\bar\partial_Bu)_h,
$$
and formula \eqref{sigma_t} implies 
\begin{equation}\label{secondvar}
\begin{split}
\partial_s\partial_t {\rm En}(f)_{|(0,0)}=\frac{1}{2^{n-1}n!}\int_M\langle i\partial_B\bar\partial_B u,\chi\rangle_h\langle i\partial_B\bar\partial_B v,\chi\rangle_h\,\eta\wedge(d\eta_h)^n,
\end{split}\nonumber
\end{equation}
which implies that $\partial_s\partial_t {\rm En}(f)_{|(0,0)}$ is positive definite as symmetric form.

\medskip 
$3.$ Given smooth curve $u\colon[0,1]\to \mathcal H_0$ in $\mathcal H_0$ and $h\in \mathcal H_0$ we denote by 
$$
\mathcal{L}(h,u) =\frac{1}{2^{n}n!}\int_0^1\int_M \dot u^2 \  \eta\wedge(d\eta_h)^n\wedge ds=(\dot u,\dot u)_{h},
$$
the square of the length of $u$ we respect to the Sasaki metric induced by $h$. 
Let $f\colon [0,\epsilon)\times [0,1]\to \mathcal H_0$ and assume that $t\mapsto f(t,s)$ is a solution to the $J$-flow for every $s\in [0,1]$. Then, by using Lemma \ref{int1}, we have
\begin{equation}
\begin{split}
2^{n}n!\partial_t\mathcal{L}(f,f)=&\partial_t\left[\int_0^1\int_M (\partial_s f)^2\,  \eta\wedge(d\eta_f)^n\wedge ds\right]\\
=&\int_0^1\int_M 2\partial_s\partial_tf\partial_s f+ (\partial_sf)^2\Delta_f(\partial_tf)\,  \eta\wedge(d\eta_f)^n\wedge ds\\
=&\int_0^1\int_M -2\partial_s\sigma_f\partial_s f- (\partial_sf)^2\Delta_f(\sigma_f)\,  \eta\wedge(d\eta_f)^n\wedge ds\\
=&-\int_0^1\left[2(\partial_s\sigma_f,\partial_sf)_f+((\partial_sf)^2,\Delta_f \sigma_f)_f \right]ds\\
=&-\int_0^1\left[2(\partial_s\sigma_f,\partial_sf)_f+(\partial_B(\partial_sf)^2,\partial_B \sigma_f)_f \right]ds\\
=&2i\int_0^1(\chi,\partial_B \partial_sf\wedge \bar \partial_B \partial_sf)_f\,ds\leq 0
\end{split}\nonumber
\end{equation}
and the equality holds if and only if $\partial_sf(t,s)$ is constant in $s$. 
\end{proof}

\section{ A Maximum principle for basic maps and tensors}\label{maxprin}
In this section we introduce a basic principle for transversally elliptic operators on Sasakian manifolds. 
The principle will be applied in the next section to compute the $C^2$-estimate about the solutions to \eqref{jflow}.

Let $(M,\xi,\Phi,\eta)$  be a Sasakian manifold.  By a {\em smooth family of  basic linear partial differential operators}
$\{E\}_{t\in [0,\epsilon)}$ we mean a smooth family of operators $E(\cdot,t)\colon C_B^{\infty}(M,\R)\to C^{\infty}_B(M,\R)$ which can be locally written as 
$$
E(h(y),t)=\sum_{1\leq |k|\leq m} a_k(y,t)
\frac{\partial^{|k|}}{\partial y^{k_1}\dots\partial y^{k_r}}h(y)
$$
for every $h\in \C_B^{\infty}(M,\R)$, where $\{y^1,\dots,y^{2n},z\}$ are real coordinates on $M$ such that $\xi=\partial_z\,. $
The maps $a_k$ are assumed to be smooth and basic in the space coordinates (see \cite{aziz} for a detailed descriptions of these operators on compact manifolds foliated by Riemannian foliations). Observe that $E$ can be regraded as a functional $E\colon C_B^{\infty}(M\times [0,\epsilon),\R)\to C_B^{\infty}(M\times [0,\epsilon),\R)$ in a natural way.  We further make the strong assumption on $E$ to satisfy 
\b\label{E}
E(h(x,t),t)\leq 0,
\e 
whenever the complex Hessian $dd^ch$ of $h$ is nonpositive at the point $(x,t)\in M\times [0,\epsilon).$ 


\begin{prop}[Maximum principle for basic maps]\label{maximum}
Assume that $h\in C^{\infty}(M\times [0,\epsilon),\mathds{R})$ satisfies
$$
\partial_th(x,t)-E(h(x,t),t)\leq 0. 
$$
Then
$$
\sup_{(x,t)\in M\times [0,\epsilon)} h(x,t)\leq \sup_{x\in M} h(x,0).
$$
\end{prop}
\begin{proof}
Fix $\epsilon_0\in (0,\epsilon)$ and let $h_\lambda\colon M\times [0,\epsilon_0]\to \R$ be the map 
$h_\lambda(x,t)=h(x,t)-\lambda t$. Assume that $h_\lambda$ achieves its global maximum at $(x_0,t_0)$ and assume by contradiction that 
$t_0>0$. Then $\partial_th_\lambda(x_0,t_0)\geq 0$ and $dd^ch_\lambda(x_0,t_0)$ is nonpositive.  Therefore condition \eqref{E} implies $E(h_\lambda(x_0,t_0),t_0)\leq 0$ and consequently 
$$
\partial_th_\lambda(x_0,t_0)-E(h_\lambda(x_0,t_0),t_0)\geq 0. 
$$
Since $\partial_th_\lambda=\partial_th-\lambda$ and $E(h_\lambda(x,t),t)=E(h(x,t),t)$, we have 
$$
0\leq \partial_th(x_0,t_0)-E(h(x_0,t_0),t_0)-\lambda \leq -\lambda,
$$
which is a contradiction. Therefore $h_\lambda$ achieves its global maximum at a point $(x_0,0)$ and  
$$
\sup_{M\times [0,\epsilon_0]} h\leq \sup_{M\times [0,\epsilon_0]} h_\lambda+\lambda\epsilon_0\leq \sup_{x\in M} h(x,0)+\lambda\epsilon_0.
$$
Since the above inequality holds for every $\epsilon_0\in (0,\epsilon)$ and $\lambda>0$, the claim follows. 
%
%
\end{proof} 

A similar result can be stated for tensors:
\begin{prop}[Maximum principle for basic tensors]\label{maxtensors}
Let $\kappa$ be a smooth curve of basic $(1,1)$-forms on $M$  for $t\in [0,\epsilon)$. Assume $\kappa$ nonpositive and such that 
$$
\partial_t\kappa_{i\bar j}(x,t)-E(\kappa_{i\bar j}(x,t),t)=N_{i\bar j}(x,t),
$$  
where $N$is a nonpositive basic form and the components are with respect to foliated coordinates. Then  $\kappa$ is nonpositive for every $t\in [0,\epsilon)$.
\end{prop}
\begin{proof}
The proof is very similar to the case of functions. We show that for every positive $\lambda$, $\kappa_{\lambda}=\kappa-t\lambda d\eta$ is nonpositive. Assume by contradiction that this is not true. Then there exists a $\lambda$, a first point $(x_0,t_0)\in M\times [0,\epsilon)$ and $g$-unitary $(1,0)$-vector $Z\in \mathcal{D}_{x_0}^{1,0}$ such that $\kappa_{\lambda}(Z,\bar Z)=0$. We extend $Z$ to a basic and unitary vector field in a small enough  neighborhood $U$ of $x$ and consider the map $f_{\lambda}\colon U\times [0,t_0]\to \R$ given by 
$f_{\lambda}=\kappa_{\lambda}(Z,\bar Z)$. Then $f_\lambda$ has a maximum at $(x_0,t_0)$ and so 
$$
\partial_tf_{\lambda}\geq 0\,,\quad E(f_{\lambda}(x_0,t_0),t_0)\leq 0,
$$
at $(x_0,t_0)$. Now since 
$$
E(f_{\lambda}(x,t),t)=E(f_0(x,t),t),
$$
we have 
$$
0\leq\partial_t(f_\lambda)=E(f_\lambda,\cdot)+N(Z,\bar Z)-\frac\lambda2\leq 0,
$$
at $(x_0,t_0)$, which implies 
$$
N(Z,\bar Z)\geq \frac\lambda2 ,
$$
at $(x_0,t_0)$, which is a contradiction. 
\end{proof}
In the following we will apply the two propositions when $E$ is the operator $\tilde \Delta_f$ depending on a smooth curve $f$ in $\mathcal{H}$ defined by: 
$$
\tilde \Delta_f(h,t)=g_f^{\bar k p}g_f^{\bar q j}\chi_{j\bar k} h_{,a\bar b}\,. 
$$

\section{Second order estimates}\label{estimates}
The following two lemmas provide the a priori estimates we need to prove the main theorem.   

\begin{lemma}\label{soe}
Let $f\colon M\times  [0,\epsilon)\to \R$ be a solution to \eqref{jflow}, with $\epsilon<\infty$. 
Then 
$$
\sigma_f\leq \min_{x\in M}\sigma_f(x,0)
$$
and there exists a  uniform constants $C$, depending only on $f_0$, such that 
$$
\gamma_f(x,t)\leq \sup_{x\in M} \gamma_f(x,0)\,{\rm e}^{C\epsilon}
$$
where $\gamma_f=\chi^{\bar j k}(g_f)_{k\bar j}$.
\end{lemma}
%
\begin{proof}
The upper bound of $\sigma_f$ easily follows from the definition of $J_{\chi}$ and Proposition \ref{maximum}. 
Indeed, differentiating  \eqref{jflow} in $t$ we have
$\ddot f=-\partial_t\sigma_f=g_f^{\bar a b}g_f^{\bar k j} \chi_{b\bar k}\dot f_{,j\bar a}=-\tilde \Delta_f\sigma_f,$ i.e.,
$$
\partial_t\sigma_f=\tilde \Delta_f\sigma_f
$$  
and Proposition \ref{maximum} implies the first inequality.
About the upper bound of $\gamma_f$, we have 
$$
\partial_t\gamma_f=\chi^{\bar j k}\partial_t\left[(g_f)_{k\bar j}\right]=\chi^{\bar j k}\dot f_{,k\bar j}.
$$
Since $f$ solves \eqref{jflow}, we have 
$$
\dot f_{,a}=g_f^{\bar k p}(g_f)_{p\bar q,a}g_f^{\bar q j}\chi_{j\bar k}-g_f^{\bar k j}\chi_{j\bar k,a}
$$
and 
\begin{equation}\label{dotfab}
\begin{split}
\dot f_{,a\bar b}=&-2g_f^{\bar k s}(g_f)_{\bar rs,\bar b}g_f^{\bar r p}(g_f)_{p\bar q,a}g_f^{\bar q j}\chi_{j\bar k}+\tilde \Delta_f[(g_f)_{a\bar b}]\\
&+g_f^{\bar k p}(g_f)_{p\bar q,a}g_f^{\bar q j}\chi_{j\bar k,\bar b}+g_f^{\bar k s}(g_f)_{\bar rs,\bar b}g_f^{\bar r j}\chi_{j\bar k,a}-g_f^{\bar k j}\chi_{j\bar k,a\bar b}.
\end{split}
\e

Let $R^T=R^T(\chi)$ be the transverse curvature of $\chi$ and ${\rm Ric}^T(\chi)$ its transverse Ricci tensor (see section \ref{preliminaries}). The components of  
$R^T$ with respect to foliated coordinates read as  $R_{j\bar k a\bar b}^T=-\chi_{j\bar k, a\bar b}+\chi^{\bar p q}\chi_{j\bar p, a}\chi_{q\bar k, \bar b}$. 

Fix a point $(x_0,t_0)\in M\times[0,\epsilon)$ and {\em special foliated coordinates} for $\chi$ around it (see subsection \ref{foliatedcoordinates}). We may further assume without loss of generality that  $(g_f)_{j\bar k}=\lambda_j\delta_{j k}$ at $(x_0,t_0)$. Then 

\begin{equation}\label{dotfab1}
\begin{split}
\dot f_{,a\bar b}=\,&\sum_{k,r=1}^n \frac{-2}{\lambda_k^2\lambda_r}(g_f)_{\bar rk,\bar b}(g_f)_{r\bar k,a} 
+\tilde \Delta_f[(g_f)_{a\bar b}]
-\sum_{k=1}^n 
\frac{1}{\lambda_k}\chi_{k\bar k,a\bar b}\quad\mbox{ at } (x_0,t_0)
\end{split}
\end{equation}
and 

\begin{equation}
\begin{split}
\partial_t\gamma_f=\sum_{a=1}^n\dot f_{,a\bar a}=
\sum_{a=1}^n\left[\sum_{k,r=1}^n \frac{-2}{\lambda_k^2\lambda_r}|(g_f)_{k\bar r,\bar a} |^2
+\tilde \Delta_f[(g_f)_{a\bar a}]
-\sum_{k=1}^n \frac{1}{\lambda_k}\chi_{k\bar k,a\bar a}\right] \quad\mbox{ at } (x_0,t_0)\,. 
\end{split}\nonumber
\end{equation}
 i.e. 
\begin{equation*}
\partial_t\gamma_f=
\sum_{a=1}^n\left(\sum_{k,r=1}^n \frac{-2}{\lambda_k^2\lambda_r}|(g_f)_{k\bar r,\bar a} |^2
+\tilde \Delta_f[(g_f)_{a\bar a}]\right)
-\sum_{k=1}^n\frac{1}{\lambda_k}{\rm Ric}^T_{k\bar k} \quad\mbox{ at } (x_0,t_0)\,.
\end{equation*}
Now a direct computation yields 
$$
\tilde \Delta_f\gamma_f=\sum_{a=1}^n \tilde \Delta_f[(g_f)_{a\bar a}]-\sum_{a,k=1}^n\frac{\lambda_a}{\lambda_k^2} R^T_{a\bar ak\bar k}  \quad\mbox{ at } (x_0,t_0)
$$
and therefore 
$$
\partial_t\gamma_f-\tilde \Delta_f\gamma_f=\sum_{a,k=1}^n\left(\sum_{r=1}^n \frac{-2}{\lambda_k^2\lambda_r}|(g_f)_{k\bar r,\bar a} |^2
+\frac{\lambda_a}{\lambda_k^2} R^T_{a\bar ak\bar k}\right)
-\sum_{k=1}^n \frac{1}{\lambda_k}{\rm Ric}^T_{k\bar k} \quad\mbox{ at } (x_0,t_0)\,.
$$
Observe that 
$$
\sum_{k=1}^n\frac1{\lambda_k}=\sigma_f(x_0,t_0)\leq C_1,\qquad \sum_{k=1}^n\lambda_k=\gamma_f(x_0,t_0),
$$
where $C_1=\min_{x\in M}\sigma_{f}(x,0)$. 
Thus for all $k=1,\dots, n$ we have 
$$
\frac1{\lambda_k}\leq C_1, \qquad \lambda_k\leq \gamma_f(x_0,t_0)\,.
$$
Since $M$ is compact, there exists a constant $C_2$ such that ${\rm Ric}^T-C_2\chi$ is nonnegative and therefore at $(x_0,t_0)$ we have
$$
|\frac{1}{\lambda_k}{\rm Ric}^T_{k\bar k}|\leq nC_1C_2\,,
\quad |\sum_{a,k=1}^n\frac{\lambda_a}{\lambda_k^2} R^T_{a\bar ak\bar k}|\leq C_1^2 |\sum_{a=1}^n\lambda_a {\rm Ric}^T_{a\bar a}|
\leq nC_1^2C_2\gamma_f,
$$
Thus there exists a constant $C$ such that
$$
\partial_t\gamma_f-\tilde \Delta_f\gamma_f\leq C\gamma_f+C.
$$
Let $F:={\rm e}^{-Ct}\gamma_f -Ct$. Then
$$
\partial_t F-\tilde \Delta_f F=e^{-Ct}\left(-c\gamma_f+\partial_t\gamma_f-\tilde \Delta \gamma_f\right)-C,
$$
and by Proposition \ref{maximum} we have  
$$
\sup_{(x,t)\in M\times [0,\epsilon)}F\leq \sup_{x\in M}F(x,0)=\sup_{x\in M} \gamma_f(x,0),
$$
which implies
$$
\sup_{(x,t)\in M\times [0,\epsilon)}\gamma_t=\sup_{x\in M} \gamma_f(x,0)e^{C\epsilon}
$$
as required. 
\end{proof}

In order to get a uniform lower bound for $d\eta_f$ we need to add an hypothesis on the bisectional curvature of $\chi$ (see Theorem \ref{bisectional} below). Observe that the existence of a uniform lower bound without further assumption would imply the existence of a critical metric in $\mathcal{H}_0$ for each choice of $\eta$ and $\chi$, in contrast with the necessary condition $\frac c2d\eta_f-\chi>0$.


\begin{theorem}\label{bisectional}
Assume that the transverse bisectional curvature of $\chi$ is nonnegative
and let $f\colon M\times [0,\epsilon)\to \R$ be a solution to \eqref{jflow}. Then there exists  constant $C$ depending only on the initial datum $f_0$ such that $C\chi-d\eta_f$ is a transverse K\"ahler form 
for every $t\in [0,\epsilon)$. 
\end{theorem}
\begin{proof}
Let  $\kappa=\frac12 d\eta_f-C\chi$ 
where $C$ is a constant chosen big enough to have $\kappa$ nonpositive  at $t=0$. 
Then $\kappa$ is a time-dependent basic $(1,1)$-form which is nonpositive at $t=0$. 
We apply  Proposition \ref{maxtensors} to show that $\kappa$ is nonpositive for every $t\in [0,\epsilon)$. 
Once a system of foliated coordinates $\{z^k,z\}$ is fixed, we have $\partial_t \kappa_{a\bar b}=\dot f_{,a\bar b}$ and formula \eqref{dotfab} implies 
\begin{equation}
\begin{split}
\partial_t \kappa_{a\bar b}=&-2g_f^{\bar k s}(g_f)_{\bar rs,\bar b}g_f^{\bar r p}(g_f)_{p\bar q,a}g_f^{\bar q j}\chi_{j\bar k}+g_f^{\bar k p}(g_f)_{p\bar q,a\bar b}g_f^{\bar q j}\chi_{j\bar k}\\
&+g_f^{\bar k p}(g_f)_{p\bar q,a}g_f^{\bar q j}\chi_{j\bar k,\bar b}+g_f^{\bar k s}(g_f)_{\bar rs,\bar b}g_f^{\bar r j}\chi_{j\bar k,a}-g_f^{\bar k j}\chi_{j\bar k,a\bar b}\\
=&-2g_f^{\bar k s}(g_f)_{\bar rs,\bar b}g_f^{\bar r p}(g_f)_{p\bar q,a}g_f^{\bar q j}\chi_{j\bar k}+\tilde\Delta\left[(g_f)_{a\bar b}\right]\\
&+g_f^{\bar k p}(g_f)_{p\bar q,a}g_f^{\bar q j}\chi_{j\bar k,\bar b}+g_f^{\bar k s}(g_f)_{\bar rs,\bar b}g_f^{\bar r j}\chi_{j\bar k,a}-g_f^{\bar k j}\chi_{j\bar k,a\bar b},
\end{split}\nonumber
\end{equation}
i.e. 
\b\label{segno1}
\begin{aligned}
\partial_t \kappa_{a\bar b}-\tilde\Delta\left[(g_f)_{a\bar b}\right]=\,&-2g_f^{\bar k s}(g_f)_{\bar rs,\bar b}g_f^{\bar r p}(g_f)_{p\bar q,a}g_f^{\bar q j}\chi_{j\bar k}\\
&+g_f^{\bar k p}(g_f)_{p\bar q,a}g_f^{\bar q j}\chi_{j\bar k,\bar b}+g_f^{\bar k s}(g_f)_{\bar rs,\bar b}g_f^{\bar r j}\chi_{j\bar k,a}-g_f^{\bar k j}\chi_{j\bar k,a\bar b}.
\end{aligned}
\e
We apply Proposition \ref{maxtensors} using as $N$ the basic form defined by the right hand part of formula \eqref{segno1}. To this end, we have to show that $N$ is nonpositive. That can be easily done as follows: fix a point $(x,t)\in M\times [0,\epsilon)$ and an arbitrary unitary vector field $Z\in \mathcal D_{x}^{1,0}$. Then we can find foliated coordinates $(z,z^{k})$ around $x$ which are special for $\chi$ and such that: $Z=\partial_{z^1|x}$ and $g_f$ takes a diagonal expression with eigenvalues $\lambda_k$ at $(x,t)$. 
Then we have
%

$$
N(Z,\bar Z)=-2\sum_{k,r=1}^n\frac{1}{\lambda_k^2\lambda_r}|(g_f)_{k\bar r,\bar 1}|^2-\sum_{k=1}^n\frac{1}{\lambda_k^2}R^T(\chi)_{k\bar k1\bar 1}
$$  
at $(x,t)$ and the claim follows. 

\end{proof}

\section{Proof of the main theorem}
The proof of Theorem $\ref{main}$ is based on the second order estimates provided in Section \ref{estimates} and on the following result in K\"ahler geometry. 
\begin{theorem}\label{kahler}
Let $B$ be an open ball about $0$ in $\C^n$ and let $\omega, \chi$ be two K\"ahler forms on $B$. Let $f\colon M\times [0,\epsilon)\to \R$ be solution to the K\"ahler $J$-flow 
$$
\dot f=c-g_f^{\bar k r}\chi_{r\bar k},
$$
where $g_f$ is the metric associated to $\omega_f=\omega+dd^cf$.
Assume that $\omega_f$ is uniformly bounded in $B\times[0,\epsilon)$. Then $f$ is $C^{\infty}$-bounded in a small ball about $0$. 
\end{theorem} 
As explained in \cite{chen}, the theorem can be proved by using the well-known Evans and Krylov's interior estimate (see \cite{gill} for a proof of the estimates in the complex case).
\begin{proof}[Proof of Theorem $\ref{main}$]
The proof of the long time existence consists in showing that every solution $f$ to \eqref{jflow} have a $C^{\infty}$-bound. Let 
$f\colon M\times [0,\epsilon_{\max})\to \mathds{R}$ be the solution to \eqref{jflow} with initial condition $f_0\in \mathcal H_0$ and assume by contradiction $\epsilon_{\max}< \infty$.  Lemma \ref{soe} implies that the second derivatives of $f$ are uniformly bounded in $M$. Since $f$ can be regarded as a collection of  solutions to the K\"ahler $J$-flow on small open balls in $\C^n$, Theorem \ref{kahler}  implies that $f$ is $C^\infty$-uniformly bounded in $M$. Therefore $f$ converges in $C^{\infty}$-norm to a smooth function $\tilde f$ as $t$ tends to $\epsilon_{\max}^{-}$. Since $\partial_t f$ is basic for every $t\in[0,\epsilon_{\max})$,  $\tilde f$ is basic and by the well-posedness of the Sasaki $J$-flow, the solution $f$ can be extended after $\epsilon_{\max}$ contradicting its maximality.  

The proof of the long time existence in the case when $\chi$ has nonnegative transverse holomorphic bisectional curvature, is obtained exactly as in the K\"ahler case. Let $f\colon M\times [0,\infty)\to \R$ be a solution to the Sasaki $J$-flow. Since $\chi$ has nonnegative holomorphic bisectional curvature, Theorem \ref{bisectional} implies that $f$ has a uniform $C^{\infty}$-bound and Ascoli-Arzel\`a implies that given a sequence  $t_j\in [0,\infty)$, $t_j\to \infty$, $f_{t_j}$ has a subsequnce converging in $C^\infty$-norm to function $f_{\infty}$ as $t_j\to \infty$. Therefore, $f$ converges to a critical map $f_{\infty}\in \mathcal{H}_0$.
\end{proof}



\end{document}